\documentclass{amsart}%
\usepackage{amsfonts}
\usepackage{amsmath}
\usepackage{amssymb}
\usepackage{graphicx}%
\providecommand{\U}[1]{\protect\rule{.1in}{.1in}}
\newtheorem{theorem}{Theorem}
\theoremstyle{plain}

\newtheorem{corollary}{Corollary}

\newtheorem{definition}{Definition}

\newtheorem{lemma}{Lemma}

\numberwithin{equation}{section}
\begin{document}
\title[Delayed perturbation of Mittag-Leffler functions]{Delayed perturbation of Mittag-Leffler functions and their applications to
fractional linear delay differential equations}
\author{N. I . Mahmudov}
\address{Department of Mathematics\\
Eastern Mediterranean University \\
Famagusta, T.R. North Cyprus\\
via Mersin 10, Turkey}
\email{nazim.mahmudov@emu.edu.tr}
\urladdr{}
\date{December 26, 1997}

\begin{abstract}
In this paper, we propose a delayed perturbation of Mittag-Leffler type matrix
function, which is an extension of the classical Mittag-Leffler type matrix
function and delayed Mittag-Leffler type matrix function. With the help of the
delayed perturbation of Mittag-Leffler type matrix function, we give an
explicit formula of solutions to linear nonhomogeneous fractional delay
differential equations.

\end{abstract}
\maketitle

\section{Introduction}

It is known that a solution of a linear system $y^{\prime}\left(  t\right)
=Ay\left(  t\right)  $, $t\geq0$ has the form $y\left(  t\right)
=e^{At}y\left(  0\right)  ,$ where exponential matrix $e^{At}$ is also called
fundamental matrix. However, it becomes more complex for seeking a fundamental
matrix for linear delay system
\begin{align}
y^{\prime}\left(  t\right)   &  =Ay\left(  t\right)  +By\left(  t-h\right)
,\ \ t\geq0\ ,h>0,\label{ld1}\\
y\left(  t\right)   &  =\varphi\left(  t\right)  ,\ \ -h\leq t\leq0,\nonumber
\end{align}
where $A,B$ are two constant square matrices. Under the assumptions that $A$
and $B$ are permutation matrices, Khusainov \& Shuklin \cite{khus1} give a
representation of a solution of a linear homogeneous system with delay by
introducing the concept of delayed matrix exponential $e_{h}^{Bt}$
corresponding to delay $h$ and matrix $B.$ They proved that fundamental matrix
of linear delay system (\ref{ld1}) (delayed perturbation of exponential matrix
$e^{At}$) can be given by $e^{At}e_{h}^{B_{1}\left(  t-h\right)  }%
,\ B_{1}=e^{-Ah}B.$ Notice that the fractional analogue of the same problem
was considered by Li and Wang \cite{wang1} in the case $A=\Theta.$ For more
recent contributions on oscillating system with pure delay, relative
controllability of system with pure delay, asymptotic stability of nonlinear
multidelay differential equations, finite time stability of differential
equations, one can refer to \cite{diblik3}-\cite{pos2} and reference therein.

Motivated by Khusainov \& Shuklin \cite{khus1}, Li and Wang \cite{wang1}, we
extend to consider representation of solutions of a fractional delay
differential equation of the form by introducing delayed perturbation of
Mittag-Leffler function%
\begin{align}
\left(  ^{C}D_{-h^{+}}^{\alpha}y\right)  \left(  t\right)   &  =Ay\left(
t\right)  +By\left(  t-h\right)  +f\left(  t\right)  ,\ \ t\in\left(
0,T\right]  ,h>0,\label{de1}\\
y\left(  t\right)   &  =\varphi\left(  t\right)  ,\ \ -h\leq t\leq0,\nonumber
\end{align}
where $\left(  ^{C}D_{-h^{+}}^{\alpha}y\right)  \left(  \cdot\right)  $ is the
Caputo derivative of order $\alpha\in\left(  0,1\right)  $, $A,B\in R^{n\times
n}$ denotes constant matrix, and $\varphi:\left[  -h,0\right]  \rightarrow
R^{n}$ is an arbitrary Caputo differentiable vector function, $f\in C\left(
\left[  -h,T\right]  ,R^{n}\right)  $, $t=lh$ for a fixed natural number $l$.

To end this section, we would like to state the main contribution as follows:

(i) We propose delayed perturbation $X_{h,\alpha,\beta}^{A,B}\left(  t\right)
$ of Mittag-Leffler type functions, by means of the matrix equations
(\ref{re1}). We show that for $B=\Theta$ the function $X_{h,\alpha,\beta
}^{A,B}\left(  t\right)  $ coincide with Mittag-Leffler type function of two
paramemters $t^{\alpha-1}E_{\alpha,\beta}\left(  At^{\alpha}\right)  $. For
$A=\Theta$ $X_{h,\alpha,\beta}^{A,B}\left(  t\right)  $ coincide with delayed
Mittag-Leffler type matrix function of two parameters $E_{h,\alpha,\beta}%
^{B}\left(  t-h\right)  .$

(ii) We explicitly represent the solution of fractional delay linear system
(\ref{de1}) via delayed perturbation of Mittag-Leffler type function.

\begin{definition}
\label{def:01}Mittag-Leffler type matrix function of two parameters
$\Phi_{\alpha,\beta}\left(  A,z\right)  :R\rightarrow R^{n\times n}$ is
defined by%
\[
\Phi_{\alpha,\beta}\left(  A,z\right)  :=z^{\beta-1}E_{\alpha,\beta}\left(
Az^{\alpha}\right)  :=z^{\beta-1}%
{\displaystyle\sum\limits_{k=0}^{\infty}}
\frac{A^{k}z^{\alpha k}}{\Gamma\left(  k\alpha+\beta\right)  },\ \ \ \alpha
,\beta>0,z\in R.
\]

\end{definition}

\begin{definition}
\label{def:11}Delayed Mittag-Leffler type matrix function of two parameters
$E_{h,\alpha,\beta}^{B}\left(  t\right)  :R\rightarrow R^{n\times n}$ is
defined by%
\begin{equation}
E_{h,\alpha,\beta}^{B}\left(  t\right)  :=\left\{
\begin{tabular}
[c]{ll}%
$\Theta,$ & $-\infty<t\leq-h,$\\
$I\frac{\left(  h+t\right)  ^{\beta-1}}{\Gamma\left(  \beta\right)  },$ &
$-h<t\leq0,$\\
$I\frac{\left(  h+t\right)  ^{\beta-1}}{\Gamma\left(  \beta\right)  }%
+B\frac{t^{\alpha+\beta-1}}{\Gamma\left(  \alpha+\beta\right)  }+B^{2}%
\frac{\left(  t-h\right)  ^{2\alpha+\beta-1}}{\Gamma\left(  2\alpha
+\beta\right)  }+...+B^{k}\frac{\left(  t-\left(  k-1\right)  h\right)
^{k\alpha+\beta-1}}{\Gamma\left(  k\alpha+\beta\right)  },$ & $\left(
k-1\right)  h<t\leq kh.$%
\end{tabular}
\ \ \ \right.  \label{ml2}%
\end{equation}

\end{definition}

In order to define delayed perturbation of Mittag-Leffler type matrix
functions, we introduce the following matrix equation for $Q_{k}\left(
s\right)  ,$ $k=1,2,...$
\begin{align}
Q_{k+1}\left(  s\right)   &  =AQ_{k}\left(  s\right)  +BQ_{k}\left(
s-h\right)  ,\nonumber\\
Q_{0}\left(  s\right)   &  =Q_{k}\left(  -h\right)  =\Theta,\ \ Q_{1}\left(
0\right)  =I,\nonumber\\
k  &  =0,1,2,...,s=0,h,2h,... \label{re1}%
\end{align}
Simple calculations show that%
\[%
\begin{tabular}
[c]{|l|l|l|l|l|l|l|}\hline
& $s=0$ & $s=h$ & $s=2h$ & $s=3h$ & $\cdots$ & $s=ph$\\\hline
$Q_{1}\left(  s\right)  $ & $I$ & $\Theta$ & $\Theta$ & $\Theta$ & $\cdots$ &
$\Theta$\\\hline
$Q_{2}\left(  s\right)  $ & $A$ & $B$ & $\Theta$ & $\Theta$ & $\cdots$ &
$\Theta$\\\hline
$Q_{3}\left(  s\right)  $ & $A^{2}$ & $AB+BA$ & $B^{2}$ & $\Theta$ & $\cdots$
& $\Theta$\\\hline
$Q_{4}\left(  s\right)  $ & $A^{3}$ & $A\left(  AB+BA\right)  +BA^{2}$ &
$AB^{2}+B\left(  AB+BA\right)  $ & $B^{3}$ & $\cdots$ & \\\hline
$\cdots$ & $\cdots$ & $\cdots$ & $\cdots$ &  & $\cdots$ & $\Theta$\\\hline
$Q_{p+1}\left(  s\right)  $ & $A^{p}$ & $\cdots$ & $\cdots$ & $\cdots$ &
$\cdots$ & $B^{p}$\\\hline
\end{tabular}
\ \ \ \
\]

\begin{definition}
\label{def:21}Delayed perturbation of two parameter Mittag-Leffler type matrix
function $X_{h,\alpha,\beta}^{A,B}$ generated by $A,B$\ is defined by%
\begin{equation}
X_{h,\alpha,\beta}^{A,B}\left(  t\right)  :=\left\{
\begin{tabular}
[c]{ll}%
$\Theta,$ & $-h\leq t<0,$\\
$I,$ & $t=0,$\\
$%
{\displaystyle\sum\limits_{i=0}^{\infty}}
{\displaystyle\sum\limits_{j=0}^{p-1}}
Q_{i+1}\left(  jh\right)  \dfrac{\left(  t-jh\right)  ^{i\alpha+\beta-1}%
}{\Gamma\left(  i\alpha+\beta\right)  },$ & $\left(  p-1\right)  h<t\leq ph.$%
\end{tabular}
\ \ \right.  \label{ml1}%
\end{equation}

\end{definition}

\begin{lemma}
Let $X_{h,\alpha,\beta}^{A,B}\left(  t\right)  $ be defined by (\ref{ml1}).
Then the following holds true:

\begin{description}
\item[(i)] if $A=\Theta$ then $X_{h,\alpha,\beta}^{A,B}\left(  t\right)
=E_{h,\alpha,\beta}^{B}\left(  t-h\right)  ,\ \ \left(  p-1\right)  h\leq
t-h\leq ph$,

\item[(ii)] if $B=\Theta$ then $X_{h,\alpha,\beta}^{A,B}\left(  t\right)
=t^{\beta-1}E_{\alpha,\beta}\left(  At^{\alpha}\right)  ,$

\item[(iii)] if $\alpha=\beta=1$ and $AB=BA$ then $X_{h,1,1}^{A,B}\left(
t\right)  =e^{At}e_{h}^{B\left(  t-h\right)  }$, $\left(  p-1\right)  h<t\leq
ph$.
\end{description}
\end{lemma}

\begin{proof}
(i) If $A=\Theta,$ then
\[
Q_{i+1}\left(  jh\right)  =\left\{
\begin{array}
[c]{c}%
\Theta,\ \ \ i\neq j,\\
B^{i},\ \ i=j,
\end{array}
\right.
\]
and $X_{h,\alpha,\beta}^{A,B}\left(  t\right)  $ coincides with $E_{h,\alpha
,\beta}^{B}\left(  t-h\right)  :$%
\begin{align*}
X_{h,\alpha,\beta}^{A,B}\left(  t\right)   &  =%
{\displaystyle\sum\limits_{i=0}^{p}}
B^{i}\dfrac{\left(  t-ih\right)  ^{i\alpha+\beta-1}}{\Gamma\left(
i\alpha+\beta\right)  }=\frac{t^{\beta-1}}{\Gamma\left(  \beta\right)
}+B\frac{\left(  t-h\right)  ^{\alpha+\beta-1}}{\Gamma\left(  \alpha
+\beta\right)  }+...+B^{p}\frac{\left(  t-ph\right)  ^{p\alpha+\beta-1}%
}{\Gamma\left(  p\alpha+\beta\right)  }\\
&  =E_{h,\alpha,\beta}^{B}\left(  t-h\right)  ,\ \ \left(  p-1\right)
h<t-h\leq ph.
\end{align*}

(ii) Trivially, from definition of $X_{h,\alpha,\beta}^{A,B}\left(  t\right)
$ we have: if $B=\Theta$, then%
\[
X_{h,\alpha,\beta}^{A,B}\left(  t\right)  =%
{\displaystyle\sum\limits_{i=0}^{\infty}}
A^{i}\frac{t^{i\alpha+\beta-1}}{\Gamma\left(  i\alpha+\beta\right)  }%
=t^{\beta-1}E_{\alpha,\beta}\left(  At^{\alpha}\right)  .
\]

(iii) It can be easily shown that $Q_{i+1}\left(  jh\right)  =\left(
\begin{array}
[c]{c}%
i\\
j
\end{array}
\right)  A^{i-j}B^{j}$. So, for $\left(  p-1\right)  h<t\leq ph$ and
$B_{1}=e^{-Ah}B$ we have
\begin{align*}
X_{h,1,1}^{A,B}\left(  t\right)   &  =%
{\displaystyle\sum\limits_{i=0}^{\infty}}
Q_{i+1}\left(  0\right)  \frac{t^{i}}{i!}+%
{\displaystyle\sum\limits_{i=1}^{\infty}}
Q_{i+1}\left(  h\right)  \frac{\left(  t-h\right)  ^{i}}{i!}+...+%
{\displaystyle\sum\limits_{i=1}^{\infty}}
Q_{i+1}\left(  \left(  p-1\right)  h\right)  \frac{\left(  t-\left(
p-1\right)  h\right)  ^{i}}{i!}\\
&  =%
{\displaystyle\sum\limits_{i=0}^{\infty}}
A^{i}\frac{t^{i}}{i!}+%
{\displaystyle\sum\limits_{i=1}^{\infty}}
\left(
\begin{array}
[c]{c}%
i\\
1
\end{array}
\right)  A^{i-1}B\frac{\left(  t-h\right)  ^{i}}{i!}+...+%
{\displaystyle\sum\limits_{i=p-1}^{\infty}}
\left(
\begin{array}
[c]{c}%
i\\
p-1
\end{array}
\right)  A^{i-p+1}B^{p-1}\frac{\left(  t-\left(  p-1\right)  h\right)  ^{i}%
}{i!}\\
&  =e^{At}+e^{A\left(  t-h\right)  }B\left(  t-h\right)  +...+%
{\displaystyle\sum\limits_{i=0}^{\infty}}
\left(
\begin{array}
[c]{c}%
i+p-1\\
p-1
\end{array}
\right)  A^{i}B^{p-1}\frac{\left(  t-\left(  p-1\right)  h\right)  ^{i+p-1}%
}{\left(  i+p-1\right)  !}\\
&  =e^{At}+e^{A\left(  t-h\right)  }B\left(  t-h\right)  +...+e^{A\left(
t-\left(  p-1\right)  h\right)  }B^{p-1}\frac{1}{\left(  p-1\right)  !}\left(
t-\left(  p-1\right)  h\right)  ^{p-1}\\
&  =e^{At}\left(  I+e^{-Ah}B\left(  t-h\right)  +...+e^{-A\left(  p-1\right)
h}B^{p-1}\frac{1}{\left(  p-1\right)  !}\left(  t-\left(  p-1\right)
h\right)  ^{p-1}\right)  =e^{At}e_{h}^{B_{1}\left(  t-h\right)  }.
\end{align*}

\end{proof}

It turns out that $X_{h,\alpha,\beta}^{A,B}\left(  t\right)  $ is a delayed
perturbation of the fundamental matrix of the equation (\ref{de1}) with $f=0.$

\begin{lemma}
$X_{h,\alpha,\alpha}^{A,B}:R\rightarrow R^{n}$ is a solution of%
\begin{equation}
^{C}D_{-h^{+}}^{\alpha}X_{h,\alpha,\alpha}^{A,B}\left(  t\right)
=AX_{h,\alpha,\alpha}^{A,B}\left(  t\right)  +BX_{h,\alpha,\alpha}%
^{A,B}\left(  t-h\right)  . \label{de4}%
\end{equation}

\end{lemma}

\begin{proof}
We verify that $X_{h,\alpha,\alpha}^{A,B}\left(  t\right)  $ satisfies
differential equation (\ref{de4}) for $t\in\left(  t_{p},t_{p+1}\right]  .$ We
adopt mathematical induction to prove our result.

(i) For $p=0$, $0<t\leq h$, we have%
\begin{align*}
X_{h,\alpha,\alpha}^{A,B}\left(  t\right)   &  =t^{\alpha-1}E_{\alpha,\alpha
}\left(  At\right)  ,\ \ X_{h,\alpha,\alpha}^{A,B}\left(  t-h\right)
=\Theta,\\
^{C}D_{-h^{+}}^{\alpha}X_{h,\alpha,\alpha}^{A,B}\left(  t\right)   &
=\ ^{C}D_{-h^{+}}^{\alpha}%
{\displaystyle\sum\limits_{i=1}^{\infty}}
A^{i}\frac{t^{\left(  i+1\right)  \alpha-1}}{\Gamma\left(  \left(  i+1\right)
\alpha\right)  }=AX_{h,\alpha,\alpha}^{A,B}\left(  t\right)  =AX_{h,\alpha
,\alpha}^{A,B}\left(  t\right)  +BX_{h,\alpha,\alpha}^{A,B}\left(  t-h\right)
.
\end{align*}

(ii) Suppose $p=n,\ \left(  n-1\right)  h<t\leq nh$ the following relation
holds:%
\[
X_{h,\alpha,\alpha}^{A,B}\left(  t\right)  =%
{\displaystyle\sum\limits_{i=0}^{\infty}}
{\displaystyle\sum\limits_{j=0}^{n-1}}
Q_{i+1}\left(  jh\right)  \frac{\left(  t-jh\right)  ^{\left(  i+1\right)
\alpha-1}}{\Gamma\left(  \left(  i+1\right)  \alpha\right)  }.
\]
Next, for $p=n+1,\ nh<t\leq\left(  n+1\right)  h$, by elementary computation,
one obtains%
\begin{align*}
^{C}D_{-h^{+}}^{\alpha}X_{h,\alpha,\alpha}^{A,B}\left(  t\right)   &  =%
{\displaystyle\sum\limits_{i=0}^{\infty}}
{\displaystyle\sum\limits_{j=0}^{n}}
Q_{i+1}\left(  jh\right)  \frac{\Gamma\left(  \left(  i+1\right)
\alpha\right)  }{\Gamma\left(  i\alpha\right)  }\frac{\left(  t-jh\right)
^{i\alpha-1}}{\Gamma\left(  \left(  i+1\right)  \alpha\right)  }\\
&  =%
{\displaystyle\sum\limits_{i=0}^{\infty}}
{\displaystyle\sum\limits_{j=0}^{n}}
Q_{i+1}\left(  jh\right)  \frac{\left(  t-jh\right)  ^{i\alpha-1}}%
{\Gamma\left(  i\alpha\right)  }=%
{\displaystyle\sum\limits_{i=0}^{\infty}}
{\displaystyle\sum\limits_{j=0}^{n}}
\left(  AQ_{i}\left(  jh\right)  +BQ_{i}\left(  jh-h\right)  \right)
\frac{\left(  t-jh\right)  ^{i\alpha-1}}{\Gamma\left(  i\alpha\right)  }\\
&  =%
{\displaystyle\sum\limits_{i=0}^{\infty}}
{\displaystyle\sum\limits_{j=0}^{n}}
AQ_{i+1}\left(  jh\right)  \frac{\left(  t-jh\right)  ^{\left(  i+1\right)
\alpha-1}}{\Gamma\left(  \left(  i+1\right)  \alpha\right)  }+%
{\displaystyle\sum\limits_{i=0}^{\infty}}
{\displaystyle\sum\limits_{j=0}^{n-1}}
BQ_{i+1}\left(  jh\right)  \frac{\left(  t-h-jh\right)  ^{\left(  i+1\right)
\alpha-1}}{\Gamma\left(  \left(  i+1\right)  \alpha\right)  }\\
&  =AX_{h,\alpha,\alpha}^{A,B}\left(  t\right)  +BX_{h,\alpha,\alpha}%
^{A,B}\left(  t-h\right)  .
\end{align*}
This ends the proof.
\end{proof}

\begin{lemma}
\label{lem:11}Let $\left(  k-1\right)  h<t\leq kh,$ $-h\leq s\leq t$. We have%
\[
\int_{s}^{t}\left(  t-r\right)  ^{-\alpha}X_{h,\alpha,\alpha}^{A,B}\left(
r-s\right)  dr=\left(  t-s-jh\right)  ^{-\alpha+i\alpha+\beta}%
{\displaystyle\sum\limits_{i=0}^{\infty}}
{\displaystyle\sum\limits_{j=0}^{p-1}}
Q_{i+1}\left(  jh\right)  \frac{\Gamma\left(  1-\alpha\right)  }{\Gamma\left(
i\alpha+\beta+1-\alpha\right)  }.
\]

\end{lemma}

\begin{theorem}
\label{thm:1}The solution $y(t)$ of (\ref{de1}) satisfying zero initial
condition, has a form%
\[
y\left(  t\right)  =\int_{-h}^{t}X_{h,\alpha,\alpha}^{A,B}\left(  t-s\right)
f\left(  s\right)  ds,\ \ t\geq0.
\]

\end{theorem}

\begin{proof}
By using the method of variation of constants, any solution of nonhomogeneous
system $y\left(  t\right)  $ should be satisfied in the form%
\begin{equation}
y\left(  t\right)  =\int_{-h}^{t}X_{h,\alpha,\alpha}^{A,B}\left(  t-s\right)
c\left(  s\right)  ds,\ \ t\geq0, \label{rp1}%
\end{equation}
where $c\left(  s\right)  ,$ $0\leq s\leq t$ is an unknown vector function and
$y(0)=0$. Having Caputo fractional differentiation on both sides of
(\ref{rp1}) , we obtain the following cases:

(i) For $0<t\leq h$ we have%
\begin{align*}
\left(  ^{C}D_{-h^{+}}^{\alpha}y\right)  \left(  t\right)   &  =Ay\left(
t\right)  +By\left(  t-h\right)  +f\left(  t\right) \\
&  =A\int_{-h}^{t}X_{h,\alpha,\alpha}^{A,B}\left(  t-s\right)  c\left(
s\right)  ds+\int_{-h}^{t-h}X_{h,\alpha,\alpha}^{A,B}\left(  t-h-s\right)
c\left(  s\right)  ds+f\left(  t\right) \\
&  =A\int_{-h}^{t}X_{h,\alpha,\alpha}^{A,B}\left(  t-s\right)  c\left(
s\right)  ds+f\left(  t\right)  .
\end{align*}
According to Lemma \ref{lem:11}, we have%
\begin{align*}
\left(  ^{C}D_{-h^{+}}^{\alpha}y\right)  \left(  t\right)   &  =\frac
{1}{\Gamma\left(  1-\alpha\right)  }\frac{d}{dt}\int_{-h}^{t}\left(
t-r\right)  ^{-\alpha}\left(  \int_{-h}^{r}X_{h,\alpha,\alpha}^{A,B}\left(
r-s\right)  c\left(  s\right)  ds\right)  dr\\
&  =\frac{1}{\Gamma\left(  1-\alpha\right)  }\frac{d}{dt}\int_{-h}^{t}c\left(
s\right)  \int_{s}^{t}\left(  t-r\right)  ^{-\alpha}%
{\displaystyle\sum\limits_{i=0}^{\infty}}
A^{i}\frac{\left(  r-s\right)  ^{\left(  i+1\right)  \alpha-1}}{\Gamma\left(
\left(  i+1\right)  \alpha\right)  }drds\\
&  =\frac{1}{\Gamma\left(  1-\alpha\right)  }\frac{d}{dt}\int_{-h}^{t}c\left(
s\right)  \int_{s}^{t}\left(  t-r\right)  ^{-\alpha}\frac{\left(  r-s\right)
^{\alpha-1}}{\Gamma\left(  \alpha\right)  }drds\\
&  +\frac{1}{\Gamma\left(  1-\alpha\right)  }%
{\displaystyle\sum\limits_{i=1}^{\infty}}
A^{i}\frac{d}{dt}\int_{-h}^{t}c\left(  s\right)  \int_{s}^{t}\left(
t-r\right)  ^{-\alpha}\frac{\left(  r-s\right)  ^{\left(  i+1\right)
\alpha-1}}{\Gamma\left(  \left(  i+1\right)  \alpha\right)  }drds
\end{align*}%
\begin{align*}
&  =c\left(  t\right)  +%
{\displaystyle\sum\limits_{i=1}^{\infty}}
\frac{1}{\Gamma\left(  1-\alpha\right)  \Gamma\left(  \left(  i+1\right)
\alpha\right)  }A^{i}\frac{d}{dt}\int_{-h}^{t}c\left(  s\right)  \left(
t-s\right)  ^{i\alpha-1}B\left(  \alpha\left(  i+1\right)  ,1-\alpha\right)
ds\\
&  =c\left(  t\right)  +%
{\displaystyle\sum\limits_{i=1}^{\infty}}
\frac{\Gamma\left(  1-\alpha\right)  \Gamma\left(  \left(  i+1\right)
\alpha\right)  }{\Gamma\left(  1-\alpha\right)  \Gamma\left(  \left(
i+1\right)  \alpha\right)  \Gamma\left(  1+i\alpha\right)  }A^{i}\frac{d}%
{dt}\int_{-h}^{t}c\left(  s\right)  \left(  t-s\right)  ^{i\alpha}ds\\
&  =c\left(  t\right)  +%
{\displaystyle\sum\limits_{i=1}^{\infty}}
A^{i}\frac{1}{\Gamma\left(  1+i\alpha\right)  }\frac{d}{dt}\int_{-h}%
^{t}c\left(  s\right)  \left(  t-s\right)  ^{i\alpha}ds
\end{align*}%
\begin{align*}
&  =c\left(  t\right)  +%
{\displaystyle\sum\limits_{i=1}^{\infty}}
A^{i}\frac{\alpha i}{\Gamma\left(  1+i\alpha\right)  }\int_{-h}^{t}c\left(
s\right)  \left(  t-s\right)  ^{i\alpha-1}ds=c\left(  t\right)  +\int_{-h}^{t}%
{\displaystyle\sum\limits_{i=1}^{\infty}}
A^{i}\frac{1}{\Gamma\left(  i\alpha\right)  }\left(  t-s\right)  ^{i\alpha
-1}c\left(  s\right)  ds\\
&  =c\left(  t\right)  +A\int_{-h}^{t}%
{\displaystyle\sum\limits_{i=0}^{\infty}}
A^{i}\frac{\left(  t-s\right)  ^{\left(  i+1\right)  \alpha-1}}{\Gamma\left(
\left(  i+1\right)  \alpha\right)  }c\left(  s\right)  ds=c\left(  t\right)
+A\int_{-h}^{t}X_{h,\alpha,\beta}^{A,B}\left(  t-s\right)  c\left(  s\right)
ds.
\end{align*}
Hence, we obtain $c(t)=f(t)$.

(ii) For $kh<t\leq\left(  k+1\right)  h$, according to (\ref{de1}) , we have%
\begin{align*}
\left(  ^{C}D_{-h^{+}}^{\alpha}y\right)  \left(  t\right)   &  =Ay\left(
t\right)  +By\left(  t-h\right)  +f\left(  t\right) \\
&  =A\int_{-h}^{t}X_{h,\alpha,\alpha}^{A,B}\left(  t-s\right)  c\left(
s\right)  ds+B\int_{-h}^{t-h}X_{h,\alpha,\alpha}^{A,B}\left(  t-h-s\right)
c\left(  s\right)  ds+f\left(  t\right) \\
&  =A\int_{-h}^{t}%
{\displaystyle\sum\limits_{i=0}^{\infty}}
{\displaystyle\sum\limits_{j=0}^{n}}
Q_{i+1}\left(  jh\right)  \frac{\left(  t-s-jh\right)  ^{\left(  i+1\right)
\alpha-1}}{\Gamma\left(  \left(  i+1\right)  \alpha\right)  }c\left(
s\right)  ds\\
&  +B\int_{-h}^{t-h}%
{\displaystyle\sum\limits_{i=0}^{\infty}}
{\displaystyle\sum\limits_{j=0}^{n-1}}
Q_{i+1}\left(  jh\right)  \frac{\left(  t-s-h-jh\right)  ^{\left(  i+1\right)
\alpha-1}}{\Gamma\left(  \left(  i+1\right)  \alpha\right)  }c\left(
s\right)  ds+f\left(  t\right) \\
&  =A%
{\displaystyle\sum\limits_{i=0}^{\infty}}
{\displaystyle\sum\limits_{j=0}^{n}}
Q_{i+1}\left(  jh\right)  \int_{-h}^{t-jh}\frac{\left(  t-s-jh\right)
^{\left(  i+1\right)  \alpha-1}}{\Gamma\left(  \left(  i+1\right)
\alpha\right)  }c\left(  s\right)  ds\\
&  +B%
{\displaystyle\sum\limits_{i=0}^{\infty}}
{\displaystyle\sum\limits_{j=1}^{n}}
Q_{i+1}\left(  jh-h\right)  \int_{-h}^{t-jh}\frac{\left(  t-s-jh\right)
^{\left(  i+1\right)  \alpha-1}}{\Gamma\left(  \left(  i+1\right)
\alpha\right)  }c\left(  s\right)  ds+f\left(  t\right) \\
&  =%
{\displaystyle\sum\limits_{i=0}^{\infty}}
{\displaystyle\sum\limits_{j=0}^{n}}
Q_{i+2}\left(  jh\right)  \int_{-h}^{t-jh}\frac{\left(  t-s-jh\right)
^{\left(  i+1\right)  \alpha-1}}{\Gamma\left(  \left(  i+1\right)
\alpha\right)  }c\left(  s\right)  ds+f\left(  t\right)  .
\end{align*}
According to Lemma \ref{lem:11}, we have%
\begin{align*}
&  \left(  ^{C}D_{-h^{+}}^{\alpha}y\right)  \left(  t\right) \\
&  =\frac{1}{\Gamma\left(  1-\alpha\right)  }\frac{d}{dt}\int_{-h}^{t}\left(
t-r\right)  ^{-\alpha}\left(  \int_{-h}^{r}X_{h,\alpha,\alpha}^{A,B}\left(
r-s\right)  c\left(  s\right)  ds\right)  dr\\
&  =\frac{1}{\Gamma\left(  1-\alpha\right)  }\frac{d}{dt}\int_{-h}^{t}c\left(
s\right)  \int_{s}^{t}\left(  t-r\right)  ^{-\alpha}X_{h,\alpha,\alpha}%
^{A,B}\left(  r-s\right)  drds\\
&  =\frac{1}{\Gamma\left(  1-\alpha\right)  }%
{\displaystyle\sum\limits_{i=0}^{\infty}}
{\displaystyle\sum\limits_{j=0}^{n}}
Q_{i+1}\left(  jh\right)  \frac{d}{dt}\int_{-h}^{t}c\left(  s\right)  \int
_{s}^{t}\left(  t-r\right)  ^{-\alpha}\frac{\left(  r-s-jh\right)  ^{\left(
i+1\right)  \alpha-1}}{\Gamma\left(  \left(  i+1\right)  \alpha\right)  }drds
\end{align*}%
\begin{align*}
&  =\frac{1}{\Gamma\left(  1-\alpha\right)  }%
{\displaystyle\sum\limits_{i=0}^{\infty}}
{\displaystyle\sum\limits_{j=0}^{n}}
Q_{i+1}\left(  jh\right)  \frac{d}{dt}\int_{-h}^{t-jh}c\left(  s\right)
\left(  t-s-jh\right)  ^{i\alpha}\frac{\Gamma\left(  1-\alpha\right)  }%
{\Gamma\left(  i\alpha+1\right)  }ds\\
&  =%
{\displaystyle\sum\limits_{j=0}^{n}}
Q_{1}\left(  jh\right)  \frac{d}{dt}\int_{-h}^{t-jh}c\left(  s\right)  ds+%
{\displaystyle\sum\limits_{i=1}^{\infty}}
{\displaystyle\sum\limits_{j=0}^{n}}
Q_{i+1}\left(  jh\right)  \int_{-h}^{t-jh}c\left(  s\right)  \left(
t-s-jh\right)  ^{i\alpha-1}\frac{1}{\Gamma\left(  i\alpha\right)  }ds\\
&  =c\left(  t\right)  +%
{\displaystyle\sum\limits_{i=0}^{\infty}}
{\displaystyle\sum\limits_{j=0}^{n}}
Q_{i+2}\left(  jh\right)  \int_{-h}^{t-jh}\frac{\left(  t-s-jh\right)
^{\left(  i+1\right)  \alpha-1}}{\Gamma\left(  \left(  i+1\right)
\alpha\right)  }c\left(  s\right)  ds.
\end{align*}

Hence, we obtain $c(t)=f(t)$. The proof is completed.
\end{proof}

\begin{theorem}
\label{thm:2}Let $p=0,1,...,l$. A solution $y\in C\left(  \left(  \left(
p-1\right)  h,ph\right]  ,R^{n}\right)  $ of (\ref{de1}) has a form%
\[
y\left(  t\right)  =X_{h,\alpha,1}^{A,B}\left(  t+h\right)  \varphi\left(
0\right)  +\int_{-h}^{0}X_{h,\alpha,\alpha}^{A,B}\left(  t-s\right)  \left(
\left(  ^{C}D_{-h^{+}}^{\alpha}\varphi\right)  \left(  s\right)
-A\varphi\left(  s\right)  \right)  ds.
\]

\end{theorem}

\begin{proof}
We looking for a solution of the form%
\[
y\left(  t\right)  =X_{h,\alpha,1}^{A,B}\left(  t+h\right)  c+\int_{-h}%
^{0}X_{h,\alpha,\alpha}^{A,B}\left(  t-s\right)  g\left(  s\right)  ds,
\]
where $c$ is an unknown constants, $g(t)$ is an unknown continuously
differentiable function. Moreover, it satisfies initial condition $y\left(
t\right)  =\varphi\left(  t\right)  $, $-h\leq t\leq0$, i.e.,%
\[
y\left(  t\right)  =X_{h,\alpha,1}^{A,B}\left(  t+h\right)  c+\int_{-h}%
^{0}X_{h,\alpha,\alpha}^{A,B}\left(  t-s\right)  g\left(  s\right)
ds:=\varphi\left(  t\right)  ,\ \ -h\leq t\leq0.
\]
Let $t=-h$ we have%
\[
X_{h,\alpha,1}^{A,B}\left(  -h-s\right)  =\left\{
\begin{tabular}
[c]{ll}%
$\Theta,$ & $-h<s\leq0,$\\
$I,$ & $s=-h.$%
\end{tabular}
\ \ \ \right.
\]
Thus $c=\varphi\left(  -h\right)  $. Since $-h\leq t\leq0$, one obtains%
\[
X_{h,\alpha,1}^{A,B}\left(  t-s\right)  =\left\{
\begin{tabular}
[c]{ll}%
$\Theta,$ & $t<s\leq0,$\\
$E_{\alpha,1}\left(  A\left(  t-s\right)  ^{\alpha}\right)  ,$ & $-h\leq s\leq
t,\ 0\leq t-s\leq t+h\leq h.$%
\end{tabular}
\ \ \ \ \ \right.
\]
Thus on interval $-h\leq t\leq0$, one can derive that%
\begin{align}
\varphi\left(  t\right)   &  =X_{h,\alpha,1}^{A,B}\left(  t+h\right)
\varphi\left(  -h\right)  +\int_{-h}^{0}X_{h,\alpha,\alpha}^{A,B}\left(
t-s\right)  g\left(  s\right)  ds\label{q1}\\
&  =X_{h,\alpha,1}^{A,B}\left(  t+h\right)  \varphi\left(  -h\right)
+\int_{-h}^{t}X_{h,\alpha,\alpha}^{A,B}\left(  t-s\right)  g\left(  s\right)
ds+\int_{t}^{0}X_{h,\alpha,\alpha}^{A,B}\left(  t-s\right)  g\left(  s\right)
ds\nonumber\\
&  =E_{\alpha,1}\left(  A\left(  t+h\right)  ^{\alpha}\right)  \varphi\left(
-h\right)  +\int_{-h}^{t}\left(  t-s\right)  ^{\alpha-1}E_{\alpha,\alpha
}\left(  A\left(  t-s\right)  ^{\alpha}\right)  g\left(  s\right)
ds.\nonumber
\end{align}
Having differentiated (\ref{q1}), we obtain%
\begin{align*}
\left(  ^{C}D_{-h^{+}}^{\alpha}\varphi\right)  \left(  t\right)   &  =A%
{\displaystyle\sum\limits_{k=0}^{\infty}}
\frac{A^{k}\left(  t+h\right)  ^{\alpha k}}{\Gamma\left(  1+k\alpha\right)
}\varphi\left(  -h\right)  +\int_{-h}^{t}%
{\displaystyle\sum\limits_{k=1}^{\infty}}
\frac{A^{k}\left(  t-s\right)  ^{\alpha k-1}}{\Gamma\left(  k\alpha\right)
}g\left(  s\right)  ds+g\left(  t\right) \\
&  =AE_{\alpha,1}\left(  A\left(  t+h\right)  ^{\alpha}\right)  \varphi\left(
-h\right)  +A\int_{-h}^{t}E_{\alpha,\alpha}\left(  A\left(  t-s\right)
^{\alpha}\right)  g\left(  s\right)  ds+g\left(  t\right) \\
&  =A\varphi\left(  t\right)  +g\left(  t\right)  .
\end{align*}
Therefore, $g\left(  t\right)  =\left(  ^{C}D_{-h^{+}}^{\alpha}\varphi\right)
\left(  t\right)  -A\varphi\left(  t\right)  $ and the desired result holds.
\end{proof}

Combining Theorems \ref{thm:1} and \ref{thm:2}, we have the following result.

\begin{corollary}
A solution $y\in C\left(  \left[  -h,T\right]  \cap\left(  \left(  p-1\right)
h,ph\right]  ,R^{n}\right)  $ of (\ref{de1}) has a form%
\begin{align*}
y\left(  t\right)   &  =X_{h,\alpha,1}^{A,B}\left(  t+h\right)  \varphi\left(
-h\right)  +\int_{-h}^{0}X_{h,\alpha,\alpha}^{A,B}\left(  t-s\right)  \left[
\left(  ^{C}D_{-h^{+}}^{\alpha}\varphi\right)  \left(  s\right)
-A\varphi\left(  s\right)  \right]  ds\\
&  +\int_{0}^{t}X_{h,\alpha,\alpha}^{A,B}\left(  t-s\right)  f\left(
s\right)  ds.
\end{align*}

\end{corollary}

\end{document}